\documentclass{article}
\usepackage[utf8]{inputenc}

\usepackage{graphicx,graphics}

\usepackage{rotating}
\usepackage[twoside]{geometry}
\usepackage{epsfig}
\usepackage{cmap}

\usepackage{graphicx,graphics}

\usepackage{hyperref}
\hypersetup{
pdftitle={AdePpaper},            
colorlinks=true,      
linkcolor=black,         
citecolor=blue}

\usepackage{amssymb,amsfonts,amstext,amsthm}
\usepackage{mathrsfs}
\usepackage[intlimits]{amsmath}
\newtheorem{proposition}{\bf Proposition}[section]

\newtheorem{lemma}{\bf Lemma}[section]
\newtheorem{theorem}{\bf Theorem}[section]

\newtheorem{remark}{Remark}[section]

\newcounter{SE}

\title{Quantitative spectral stability and sharp Fourier estimates for periodic CMV matrices with decaying perturbations}
\author{Moacir Aloisio, C\'esar R. de Oliveira and Mariane Pigossi}

\begin{document}

\maketitle

\begin{abstract}
We study half-line CMV matrices whose Verblunsky coefficients are decaying perturbations of periodic sequences. We quantify the stability of fine properties of the absolutely continuous spectral measure under such perturbations. More precisely, we establish
H\"older continuity of the perturbed spectral measure and derive sharp time-averaged Fourier estimates governing the associated quantum dynamics. These results remain valid up to the edges of the absolutely continuous spectrum. Since the spectral measure of a half-line CMV matrix is precisely the orthogonality measure for the corresponding
orthogonal polynomials on the unit circle (OPUC), our results also yield new statements for OPUC whose Verblunsky coefficients are  decaying perturbations of periodic sequences.
\end{abstract}

\ 


\

\noindent{\bf  AMS classification codes}: 28A80 (primary), 42A85 (secondary).    

\renewcommand{\thetable}{\Alph{table}}


\section{Contextualization and main result}

We study quantitative properties of absolutely continuous spectral measures associated with half-line CMV matrices whose Verblunsky
coefficients are  decaying perturbations of periodic sequences. In particular, we establish critical H\"older continuity
of the spectral measure and sharp time-averaged Fourier estimates related to discrete-time quantum return probabilities.

Let
\[
\mathbb D=\{z\in\mathbb C: |z|<1\},
\, \, 
\partial\mathbb D=\{z\in\mathbb C: |z|=1\}.
\]
Given a sequence of Verblunsky coefficients
$\alpha=\{\alpha_n\}_{n\geq0}\subset\mathbb D$, set
\[
\rho_n=\left(1-|\alpha_n|^2\right)^{1/2},
\, \,  n\geq0.
\]
The half-line CMV matrix associated with $\alpha$ is the unitary
operator
\[
\mathcal C_\alpha:\ell^2(\mathbb N_0)\to \ell^2(\mathbb N_0)
\]
whose matrix representation in the canonical basis
$\{\delta_n\}_{n\geq0}$ is
\[
\mathcal C_\alpha =
\begin{pmatrix}
\overline{\alpha}_0 & \overline{\alpha}_1\rho_0 & \rho_1\rho_0
& 0 & 0 & 0 & \cdots \\
\rho_0 & -\overline{\alpha}_1\alpha_0 & -\rho_1\alpha_0
& 0 & 0 & 0 & \cdots \\
0 & \overline{\alpha}_2\rho_1 & -\overline{\alpha}_2\alpha_1
& \overline{\alpha}_3\rho_2 & \rho_3\rho_2 & 0 & \cdots \\
0 & \rho_2\rho_1 & -\rho_2\alpha_1
& -\overline{\alpha}_3\alpha_2 & -\rho_3\alpha_2 & 0 & \cdots \\
0 & 0 & 0 & \overline{\alpha}_4\rho_3
& -\overline{\alpha}_4\alpha_3 & \overline{\alpha}_5\rho_4
& \rho_5\rho_4 & \cdots \\
0 & 0 & 0 & \rho_4\rho_3
& -\rho_4\alpha_3 & -\overline{\alpha}_5\alpha_4
& -\rho_5\alpha_4 & \cdots \\
\vdots & \vdots & \vdots & \vdots & \vdots & \vdots & \ddots
\end{pmatrix}.
\]

Throughout the paper, we consider half-line CMV matrices on $\ell^2(\mathbb N_0)$, with the canonical boundary condition at the
origin, and for each $\mathcal C_\alpha$ let $\mu_\alpha$ denote its spectral measure associated with the cyclic vector $\delta_0$.
Equivalently, $\mu_\alpha$ is the unique probability measure on
$\partial\mathbb D$ satisfying
\[
\left\langle \delta_0,\mathcal C_\alpha^k\delta_0\right\rangle
=
\int_{\partial\mathbb D} z^k\,d\mu_\alpha(z),
\, \,  k\in\mathbb Z.
\]

CMV matrices are the unitary analogues of Schr\"odinger and Jacobi operators and play a central role in the spectral theory of orthogonal polynomials on the
unit circle. Periodic Verblunsky coefficients form one of the most classical classes in this setting: we say that
$\alpha=\{\alpha_n\}_{n\geq0}$ is periodic if there exists an integer $p\geq1$, called the period, such that
$\alpha_{n+p}=\alpha_n$ for all $n\geq0$. The associated essential spectrum is a finite union of closed arcs of the unit circle, and the
corresponding spectral measures are purely absolutely continuous on the essential spectrum, up to possible singular behavior at the band
edges \cite{DamanikFillmanVance,GuoDamanikOng,MungerOng,SimonOPUC1}. 

In this work we consider perturbations of the form
\[
\alpha_n^\kappa=\alpha_n+\kappa b(n),
\, \,  n\geq0,
\]
where $\alpha=\{\alpha_n\}_{n\geq0}$ is periodic, $b$ decays at least polynomially, sufficiently fast, and $\kappa$ is sufficiently small. Polynomial decay implies that the perturbation is compact, and hence the essential spectrum is preserved \cite{SimonOPUC1,SimonOPUC2}. However, preservation of the essential spectrum does not determine the spectral type inside the bands.
Likewise, even when the absolutely continuous component is stable, qualitative perturbation theory does not provide quantitative control
of the spectral measure, particularly at the band edges. The main result below establishes purely absolutely continuous spectrum
together with critical $1/2$-H\"older continuity and sharp time-averaged Fourier estimates up to the spectral edges.

We present our main result.

\begin{theorem} \label{thm:perturbed-periodic-cmv}
Let $0<\beta <1$ and let $\{\alpha_n\}_{n\geq0}\subset\mathbb D$ be a
periodic sequence of Verblunsky coefficients such that
\[
\sup_{n\geq0}|\alpha_n|< \beta<1,
\]
and let $\mathcal C_\alpha$ be the associated half-line CMV matrix
acting on $\ell^2(\mathbb N_0)$. Let $\gamma>0$, $\tau>5$, and let $b:\mathbb N_0\to\mathbb C$ satisfy
\[
|b(n)|\leq \gamma (n+1)^{-\tau}, \, \,  n\geq 0,
\]
(in particular $b\in\ell^\infty(\mathbb N_0)$). For $\kappa\in\mathbb C$ with $|\kappa|$ sufficiently small,
set
\[
\alpha_n^\kappa=\alpha_n+\kappa b(n),
\, \,   n\geq0.
\]
Write $\mathcal C_\kappa=\mathcal C_{\alpha^\kappa}$ for the corresponding perturbed CMV matrix and $\mu_{\alpha^\kappa}$ for its spectral measure associated with $\delta_0$.  Then there exists $\varepsilon>0$ such that, for every
$|\kappa|<\varepsilon$, the half-line CMV matrix $\mathcal C_\kappa$
satisfies:
\begin{enumerate}
\item[(i)] $
\mathfrak e
=
\sigma_{\mathrm{ess}}(\mathcal C_\kappa)
=
\sigma_{\mathrm{ess}}(\mathcal C_\alpha),$ 
and the spectral measure of $\mathcal C_\kappa$ is purely absolutely
continuous on $\mathfrak e$;

\item[(ii)]
the spectral measure $\mu_{\alpha^\kappa}|_{\mathfrak e}$ is
$\frac12$-H\"older continuous on $\mathfrak e$;

\item[(iii)]
for
\[
\psi
=
\frac{\chi_{\mathfrak e}(\mathcal C_\kappa)\delta_0}
{\|\chi_{\mathfrak e}(\mathcal C_\kappa)\delta_0\|}
\]
and for every integer $K\geq2$,
\[
\frac1K
\sum_{n=0}^{K-1}
\left|
\left\langle
\psi,\mathcal C_\kappa^n\psi
\right\rangle
\right|^2
\lesssim
\frac{\log K}{K},
\]
where $\chi_{\mathfrak e}(\mathcal C_\kappa)$ is defined through the
spectral functional calculus ($\chi_{\mathfrak e}$ is the characteristic function of the set~$\mathfrak e$).
\end{enumerate}
\end{theorem}

The proof of Theorem \ref{thm:perturbed-periodic-cmv} is based on a comparison between the Szeg\H{o} transfer matrices of the periodic and perturbed CMV matrices, combined with the Bernstein-Szeg\H{o} approximation and periodic Floquet theory \cite{SimonOPUC1,SimonOPUC2}. This transfers the square-root band-edge behavior of the periodic
spectral measure to the perturbed one, yielding both the critical $1/2$-H\"older regularity and the sharp $\log K/K$ Fourier estimate. 

Although the general perturbative philosophy is inspired by the one-dimensional Schr\"odinger setting in \cite{AloCardeO}, the proof developed here follows a fundamentally different route. Its core arguments are specific to CMV matrices and rely on structural properties of Szeg\H{o} transfer matrices, in particular their
$J$-unitary geometry (see Lemma~\ref{teclemma} ahead). Thus, the proof is not obtained by translating the self-adjoint argument to the unitary setting, but by building a distinct mechanism adapted to the CMV framework.

A key point in the proof of Theorem~\ref{thm:perturbed-periodic-cmv} is that the asymptotic comparison between the Szeg\H{o} transfer matrices of the periodic and perturbed
CMV matrices  produces a vector which depends on the spectral
parameter (see Proposition~\ref{prop:asymptotic-formula-cmv} ahead). This prevents a direct application of the
classical Bernstein--Szeg\H{o} approximation  (Theorem \ref{thm:cmv-carmona}). The $J$-unitary geometry of the Szeg\H{o} transfer matrices is precisely what allows us to overcome this difficulty and recover an appropriate periodic comparison, as developed in Lemma~\ref{teclemma} (see Section~2) and used in Section~4.

\begin{remark}
\end{remark}

\begin{enumerate}
    \item [i)] The first  conclusion in the statement of Theorem \ref{thm:perturbed-periodic-cmv} is compatible with the general perturbation theory for CMV matrices. Indeed,  polynomial decay in the statement of Theorem \ref{thm:perturbed-periodic-cmv}  implies
\[
\sum_{n\geq0}|\alpha_n^\kappa-\alpha_n|<\infty,
\]
and hence the corresponding CMV matrices differ by a trace-class operator. Classical trace-class perturbation theory yields invariance,
up to unitary equivalence, of the absolutely continuous parts \cite{SimonOPUC1,SimonOPUC2}. It does not, however, exclude singular
spectrum inside the bands or provide quantitative control of the canonical spectral measure at the band edges. Theorem~\ref{thm:perturbed-periodic-cmv}
therefore supplies information that is not contained in the standard qualitative perturbation theory. It is worth underlining that Jost and meromorphic approaches such as \cite{Kozhan2012,Kozhan2014} do not provide this type of quantitative band-edge analysis (Theorem~\ref{thm:perturbed-periodic-cmv} $(ii)$-$(iii)$), even for the unperturbed periodic problem.

    \item [ii)] Theorem~\ref{thm:perturbed-periodic-cmv} $(iii)$ is polynomially optimal. Namely, the upper and lower correlation dimensions of
$\mu^\kappa_\alpha\big|_{\mathfrak e}$, denoted by $D^{+}(\mu^\kappa_\alpha\big|_{\mathfrak e})$ and $D^{-}(\mu^\kappa_\alpha\big|_{\mathfrak e})$, govern the power-law decay of the time-averaged quantum return probability through the exponents $K^{-D^{\mp}}$ \cite{BGT,AloCarvCRdOlivro}. Moreover, since $\mu^\kappa_\alpha\big|_{\mathfrak e}$ is a finite Borel measure on the unit circle, its upper and lower correlation dimensions satisfy
\[
0\leq
D^\pm\!\left(\mu^\kappa_\alpha\big|_{\mathfrak e}\right)
\leq 1.
\]
We also note that the critical $1/2$-H\"older continuity alone, however, does not imply the sharp estimate in part~$(iii)$: the
generic bound obtained from $\alpha$-H\"older continuity via a classical Strichartz-type argument \cite{DamanikFillmanVance} is only $O(K^{-\alpha})$, which gives
$O(K^{-1/2})$ when $\alpha=1/2$. To obtain the much stronger $\log K/K$ rate, we prove a sharp Strichartz-type inequality, refining a result of \cite{DamanikFillmanVance}.

    \item [iii)] We restrict our analysis to the essential spectrum, since eigenvalues,
and therefore atomic components of the spectral measure, may occur
outside the spectral bands.

    \item [iv)] Avila and Jitomirskaya observed in Remark~4.4 in \cite{AvilaUaH} that one should expect quantitative estimates for the modulus of absolute continuity of spectral measures, with the critical exponent $1/2$ dictated by square-root singularities at spectral gap edges. To the best of our knowledge, there have been very few advances in recent decades in the understanding of this problem. We highlight \cite{AloCardeO,AvilaUaH,kreimerLS2009,LiYouZhou2026,lukic2019}. The main difficulty addressed here is the quantitative analysis of the spectral measure uniformly up to the spectral edges.

    \item[v)]  We note that quantitative regularity of naturally arising absolutely continuous measures remains a subtle problem in dynamical systems as well. For stationary measures, recent work such as \cite{GorodetskiK} establishes H\"older regularity. It is worth underlining that, in the context of physical measures, the exponent $1/2$ is expected to play a critical role. In the present spectral setting, this exponent arises naturally from square-root singularities at band edges.

\end{enumerate}

Orthogonal polynomials on the unit circle form a classical and extensively developed branch of spectral theory, with a rich theory of
periodic operators, perturbations, and asymptotic behavior; see, in particular, \cite{DamanikKillipSimon,SimonOPUC1,SimonOPUC2}. The canonical spectral measure of a half-line CMV matrix is precisely the orthogonality measure of the corresponding OPUC. Thus, Theorem~\ref{thm:perturbed-periodic-cmv} has the following direct OPUC formulation.

\begin{theorem}[OPUC formulation of
Theorem~\ref{thm:perturbed-periodic-cmv}]
\label{mainthm}
Under the assumptions of Theorem~\ref{thm:perturbed-periodic-cmv}, let
$\mu_{\alpha^\kappa}$ be the orthogonality measure on $\partial\mathbb D$
associated with the Verblunsky coefficients
\[
\alpha_n^\kappa=\alpha_n+\kappa b(n).
\]
Then, for $|\kappa|$ sufficiently small, the restriction
$\mu_{\alpha^\kappa}|_{\mathfrak e}$, where
\[
\mathfrak e=\sigma_{\rm ess}(\mathcal C_\alpha),
\]
is purely absolutely continuous and $1/2$-H\"older continuous.
Moreover,
\[
\frac1K\sum_{n=0}^{K-1}
\left|
\int_{\mathfrak e}z^{-n}\,d\mu_{\alpha^\kappa}(z)
\right|^2
\lesssim
\frac{\log K}{K},
\, \,  K\ge2.
\]
\end{theorem}

We use the notation
\[
X\lesssim Y
\]
to mean that $X\leq \gamma Y$ for a constant $\gamma>0$ independent
of the asymptotic variables under consideration. We write
\[
X \asymp Y
\]
if \(X \lesssim Y\) and \(Y \lesssim X\).

The paper is organized as follows. In Section~2 we collect the
transfer-matrix and spectral-measure ingredients needed later,
including the Bernstein-Szeg\H{o} approximation
(Theorem~\ref{thm:cmv-carmona}), the relevant periodic Floquet theory,
and the asymptotic comparison between the periodic and perturbed
Szeg\H{o} transfer matrices (Proposition~\ref{prop:asymptotic-formula-cmv}). Section~3 is devoted to a sharp Strichartz-type (Proposition \ref{thmStrichartz}) estimate for measures with power-law singularities. In Section~4 we combine these ingredients to prove Theorem~\ref{thm:perturbed-periodic-cmv}.


\section{Szeg\H{o} transfer matrices and spectral measures}

In this work we study half-line CMV matrices acting
on $\ell^2(\mathbb N_0)$, where
$\mathbb N_0=\{0,1,2,\ldots\}$. Let
\[
\mathbb D=\{z\in\mathbb C:|z|<1\},
\, \, 
\partial\mathbb D=\{z\in\mathbb C:|z|=1\},
\]
and let $\ell$ denote normalized Lebesgue measure on $\partial\mathbb D$.
A sequence $\alpha=(\alpha_j)_{j\geq0}$ of Verblunsky coefficients
satisfies
\[
|\alpha_j|<1,
\, \, 
\rho_j=\sqrt{1-|\alpha_j|^2}. \, \, j \geq 0.
\]
The associated half-line CMV matrix $\mathcal C_\alpha$ is the
unitary operator on $\ell^2(\mathbb N_0)$ obtained from the standard
CMV construction with the canonical boundary condition at $0$.

For $z\in\mathbb C\setminus\{0\}$, the corresponding Szeg\H{o}
transfer matrix is
\[
S(\alpha,z)
=
\frac{1}{\rho}
\begin{pmatrix}
z & -\overline{\alpha}\\
-\alpha z & 1
\end{pmatrix},
\, \, 
\rho=\sqrt{1-|\alpha|^2}.
\]
We write
\[
A_\alpha(n,z)
=
S(\alpha_{n-1},z)\cdots S(\alpha_0,z),
\, \, 
A_\alpha(0,z)=I.
\]
The spectral measure of $\mathcal C_\alpha$ associated with the
canonical vector $\delta_0$ will be denoted by $\mu_\alpha$.

The following form of the Bernstein-Szeg\H{o} approximation will be
used repeatedly.

\begin{theorem}[Bernstein-Szeg\H{o} approximation]
\label{thm:cmv-carmona}
Let $\alpha$ be a sequence of Verblunsky coefficients and let
\[
u_0=\frac{1}{\sqrt2}
\begin{pmatrix}
1\\
1
\end{pmatrix}.
\]
Then, for every $f\in C(\partial\mathbb D)$,
\[
\lim_{n\to\infty}
\int_{\partial\mathbb D}
f(z)\,
\|A_\alpha(n,z)u_0\|^{-2}\,d\ell(z)
=
\int_{\partial\mathbb D}
f(z)\,d\mu_\alpha(z).
\]
\end{theorem}

\begin{remark} { \rm The preceding result is the transfer-matrix formulation of the
classical Bernstein-Szeg\H{o} approximation. In the OPUC notation,
one has
\[
A_\alpha(n,z)u_0
=
\frac{1}{\sqrt2}
\begin{pmatrix}
\varphi(z,n)\\
\varphi^*(z,n)
\end{pmatrix},
\]
where $\varphi(\,\cdot\,,n)$ and $\varphi^*(\,\cdot\,,n)$ are the
orthonormal polynomial and its reversed polynomial, respectively.
On $\partial\mathbb D$,
\[
|\varphi(z,n)|=|\varphi^*(z,n)|,
\]
and consequently
\[
\|A_\alpha(n,z)u_0\|^{-2}
=
|\varphi^*(z,n)|^{-2}.
\]
Thus the theorem is equivalent to the usual Bernstein-Szeg\H{o} 
approximation formula (see, e.g,  Appendix~A.2, Section~C, in~\cite{SimonOPUC1}) for the orthogonality measure on the unit
circle.}
\end{remark}

We recall some standard facts about periodic CMV matrices and Floquet theory; see, for instance, \cite{DamanikKillipSimon,SimonOPUC2,Simonmain}. Let $\alpha=\{\alpha_n\}_{n\geq0}$ be a periodic sequence with even period $p$. The periodic CMV discriminant is given by
\[
\Delta(z)= z^{-p/2}\operatorname{tr}\bigl(A_\alpha(p,z)\bigr),
\, \,  z\in\partial\mathbb D.
\]
The function $\Delta$ is real-valued on $\partial\mathbb D$, and Floquet theory yields
\[
\sigma_{\mathrm{ess}}(\mathcal C_\alpha)
=
\{z\in\partial\mathbb D: |\Delta(z)|\leq2\}.
\]
Thus, the essential spectrum is a finite union of closed arcs of the unit
circle, called spectral bands:
\[
\sigma_{\mathrm{ess}}(\mathcal C_\alpha)
=
\bigcup_{j=1}^{q} B_j.
\]
For every $z\in B_j^\circ$, there exists
$\theta_j(z)\in(0,\pi)$ such that
\[
\Delta(z)=2\cos\theta_j(z).
\]
The endpoints of the spectral bands satisfy $|\Delta(z)|=2$. Note that there is no loss of generality in assuming that the period \(p\) is even, since a sequence with odd period (p) is also periodic with period \(2p\).

We next record a remark concerning the periodic setting and the subsequent comparison argument.

\begin{remark}\label{Remark} {\rm For periodic Verblunsky coefficients, the absolutely continuous part of the spectral measure admits a useful representation. The family
$\{\mu_{\lambda\alpha}\}_{\lambda\in\partial\mathbb D}$ may be viewed as an Aleksandrov family. If $f$ denotes the Schur function associated with the
sequence $\alpha$, then the Schur function corresponding to
$\lambda\alpha$ satisfies
\[
    f_{\lambda}(z)=\lambda f(z),
\]
and therefore the corresponding Carath\'eodory function is
\[
    F_{\lambda}(z)
    =
    \frac{1+\lambda z f(z)}
         {1-\lambda z f(z)}.
\]
Hence, for almost every $z\in\partial\mathbb D$,
\[
    \frac{d\mu_{\lambda\alpha,\mathrm{ac}}}{d\ell}(z)
    =
    \Re F_{\lambda}(z)
    =
    \frac{1-|f(z)|^2}
         {|1-\lambda z f(z)|^2}.
\]
This representation is particularly convenient because the dependence on $\lambda\in\partial\mathbb D$ is explicit; see the discussion of Aleksandrov measures in Subsection~1.3.9 in \cite{SimonOPUC1}.}
\end{remark}

\begin{lemma}
\label{teclemma}
Let $\alpha$ be periodic, let
\[
\mathfrak e=\sigma_{\rm ess}(\mathcal C_\alpha),
\, \, 
J=
\begin{pmatrix}
1&0\\
0&-1
\end{pmatrix},
\]
and let
\[
u:\mathfrak e\longrightarrow\mathbb C^2
\]
be continuous and satisfy
\[
\|u(z)\|=1,
\, \, 
u(z)^*Ju(z)=0,
\, \, 
z\in\mathfrak e.
\]
Then there exists a nonnegative density
\[
\rho_{\alpha,u}\in L^1(\mathfrak e,d\ell)
\]
such that, for every
$g\in C_c(\mathfrak e^\circ)$,
\[
\lim_{n\to\infty}
\int_{\partial\mathbb D}
g(z)\,
\|A_\alpha(n,z)u(z)\|^{-2}\,d\ell(z)
=
\int_{\partial\mathbb D}
g(z)\rho_{\alpha,u}(z)\,d\ell(z).
\]
Moreover, $\rho_{\alpha,u}$ is bounded on compact subsets of
$\mathfrak e^\circ$ and, at every band edge, has at most a
square-root singularity.
\end{lemma}

\begin{proof} It suffices to consider $g\in C_c(B_j^\circ)$ for a fixed open band $B_j^\circ$, since the general case $g\in C_c(\mathfrak e^\circ)$ follows by decomposing $g$ into a finite sum of functions supported in the open bands. Set
\[
K=\operatorname{supp}g\Subset B_j^\circ.
\]
By the standard Floquet theory for periodic Verblunsky coefficients (see, e.g., Chapter~11 in \cite{SimonOPUC2}), there exists $\gamma_K>1$ such
that
\begin{equation}\label{eq:uniform-periodic-bounds}
\gamma_K^{-1}
\leq
\|A_\alpha(n,z)v\|
\leq
\gamma_K,
\, \, 
n\geq0,\ z\in K,\ \|v\|=1.
\end{equation}

Since $u(z)^*Ju(z)=0$ and $\|u(z)\|=1$, we have
\[
|u_1(z)|=|u_2(z)|=\frac1{\sqrt2}.
\]
Hence, we can write 
\[
u(z)=\eta(z)u_{\lambda(z)} 
:= \eta(z)\frac1{\sqrt2}
\begin{pmatrix}
1\\
\lambda(z)
\end{pmatrix},
\, \,  |\eta(z)|=|\lambda(z)|=1,
\]
where $\eta,\lambda:K\to\partial\mathbb D$ are given by
$\eta(z)=\sqrt2\,u_1(z)$ and
$\lambda(z)=\displaystyle\frac{u_2(z)}{u_1(z)}$, respectively, and are continuous.

Since $|\eta(z)|=1$, we have
\[
\|A_\alpha(n,z)u(z)\|
=
\|A_\alpha(n,z)u_{\lambda(z)}\|.
\]

We first consider a fixed $\lambda\in\partial\mathbb D$. Let
\[
D_\lambda=
\begin{pmatrix}
1&0\\
0&\lambda
\end{pmatrix}.
\]
A direct calculation gives
\[
S(\alpha,z)D_\lambda
=
D_\lambda S(\overline{\lambda}\alpha,z),
\]
and therefore
\begin{equation}\label{eq:phase-transfer}
A_\alpha(n,z)u_\lambda
=
D_\lambda
A_{\overline{\lambda}\alpha}(n,z)u_0.
\end{equation}
Since $D_\lambda$ is unitary,
\[
\|A_\alpha(n,z)u_\lambda\|
=
\|A_{\overline{\lambda}\alpha}(n,z)u_0\|.
\]
Theorem~\ref{thm:cmv-carmona} thus yields
\begin{equation}\label{eq:fixed-lambda-limit}
\|A_\alpha(n,z)u_\lambda\|^{-2}\,d\ell(z)
\overset{w}{\longrightarrow}
d\mu_{\overline{\lambda}\alpha}(z).
\end{equation}

Note also that, for every $p \in \mathbb{N}$, \eqref{eq:phase-transfer} implies
\[
A_{\overline{\lambda}\alpha}(p,z)
=
D_\lambda^{-1}A_\alpha(p,z)D_\lambda,
\]
so that
\[
\Delta_{\overline{\lambda}\alpha}(z)=\Delta_\alpha(z).
\]
Hence, all the measures $d\mu_{\overline{\lambda}\alpha}$ have the same
spectral bands. By the standard periodic CMV theory (see Remark \ref{Remark} above), on each band one may write
\begin{equation*}
d\mu_{\overline{\lambda}\alpha}(z)
=
\frac{1-|f(z)|^2}
         {|1-\overline{\lambda} z f(z)|^2} \, d \ell (z),
\, \,  z\in B_j^\circ.
\end{equation*}
These densities have singularities of order at most $1/2$ at open band edges; closed-gap edges produce no worse singularity.

We now approximate the continuous function $\lambda$ uniformly on $K$
by step functions
\[
\lambda_N(z)
=
\sum_{m=1}^{M_N}\lambda_m^{(N)}
\chi_{I_m^{(N)}}(z),
\, \, 
|\lambda_m^{(N)}|=1,
\]
and set $u_N(z)=\eta(z) u_{\lambda_N(z)},$ where $\eta$ was previously defined. Then
\[
\sup_{z\in K}\|u_N(z)-u(z)\|\longrightarrow0 \text{ as } N \to \infty.
\]

For each fixed $N$, \eqref{eq:fixed-lambda-limit} applied on the finitely many arcs $I_m^{(N)}$ gives

\begin{eqnarray*}
\lim_{n\to\infty}
\int_{\partial\mathbb D}
g(z)\|A_\alpha(n,z)u_N(z)\|^{-2}\,d\ell(z) &=&  \lim_{n\to\infty}
\int_{\partial\mathbb D}
g(z)\|A_\alpha(n,z)u_{\lambda_N(z)}\|^{-2}\,d\ell(z)
\\&=&
\int_{\partial\mathbb D}
g(z)\rho_{\alpha,u_N}(z)\,d\ell(z),    
\end{eqnarray*}
where
\[
\rho_{\alpha,u_N}(z)
=
\frac{1-|f(z)|^2}
         {|1-\overline{\lambda_N}(z) z f(z)|^2}.
\]
Here one may use the weak convergence \eqref{eq:fixed-lambda-limit} with $f\chi_{I_m^{(N)}}$, since its only discontinuities are the endpoints of the arcs and the periodic limiting measures are absolutely continuous in $B_j^\circ$.

On the other hand, by \eqref{eq:uniform-periodic-bounds},
\[
\left|
\|A_\alpha(n,z)u_N(z)\|^{-2}
-
\|A_\alpha(n,z)u(z)\|^{-2}
\right|
\le
2\gamma_K^6\|u_N(z)-u(z)\|.
\]
Hence, the uniform convergence $u_N\to u$ implies
\begin{equation}\label{limuni}
\sup_{z\in K}
\left|
\|A_\alpha(n,z)u_N(z)\|^{-2}
-
\|A_\alpha(n,z)u(z)\|^{-2}
\right|
\longrightarrow 0    
\end{equation}
as $N \to \infty$ uniform in $n.$  We also have, for every $z \in K$,
\[
\rho_{\alpha,u_N} (z) \longrightarrow\rho_{\alpha,u} (z)\, \, \text{as} \, N \to \infty,
\]
where
\[
\rho_{\alpha,u}(z)
=
\frac{1-|f(z)|^2}
         {|1-\overline{\lambda}(z) z f(z)|^2}.
\]
Passing first to the limit in $n$ and then in $N$ and applying the dominated convergence theorem, we have
\[
\lim_{n\to\infty}
\int_{\partial\mathbb D}
g(z)\|A_\alpha(n,z)u(z)\|^{-2}\,d\ell(z)
=
\int_{\partial\mathbb D}
g(z)\rho_{\alpha,u}(z)\,d\ell(z),
\]
where above we are using the fact that the limit \eqref{limuni}, as $N\to\infty$, is uniform in $n$. It remains to control the behavior at the band edges. Let $z_0$ be an open
band edge. Since
\[
\rho_{\alpha,u}(z)
=
\frac{1-|f(z)|^2}{|1-\overline{\lambda}(z)zf(z)|^2},
\]
and $|z|=|\overline{\lambda}(z)|=1$, we have
\[
|1-\overline{\lambda}(z)zf(z)|\geq 1-|f(z)|.
\]
Hence,
\[
\rho_{\alpha,u}(z)
\leq
\frac{1+|f(z)|}{1-|f(z)|}.
\]
For periodic Verblunsky coefficients,
\[
1-|f(z)|
\asymp
\sqrt{4-\Delta(z)^2}
\]
near every open band edge. Namely, using the Schur--Carath\'eodory relation
\[
zf(z)=\frac{F(z)-1}{F(z)+1},
\]
we obtain
\[
1-|f(z)|^2
=
\frac{4\Re F(z)}{|F(z)+1|^2}.
\]
For periodic Verblunsky coefficients, the quadratic equation satisfied by $F$ has discriminant $4(\Delta^2-4)$; see
\cite[Theorem~11.3.2 and its proof, pp.~726--727]{SimonOPUC2}. The resulting periodic formulas imply, in both the resonant and non-resonant cases,
\[
1-|f(z)|^2 = \frac{4\Re F(z)}{|F(z)+1|^2}\asymp \sqrt{4-\Delta(z)^2}.
\]
Since
\[
1-|f(z)|
=
\frac{1-|f(z)|^2}{1+|f(z)|},
\]
and $1\leq 1+|f(z)|\leq 2$, for every $z$,
\[
\rho_{\alpha,u}(z)
\lesssim
\frac{1}{\sqrt{4-\Delta(z)^2}}
\lesssim
|z-z_0|^{-1/2}.
\]
This proves the theorem.
\end{proof}

Next, we present the quantitative transfer-matrix comparison that will be used in the proof of the main result. The main ingredient is an a priori bound on the perturbed transfer
matrices, obtained by combining a telescoping argument with a discrete Gronwall inequality. This is analogous in spirit to the argument used
for perturbed Schr\"odinger cocycles in \cite{DamanikImp}; see also Theorem 2.2 in \cite{AloCardeO}.

Throughout the remainder of this section, we denote
\[
\mathfrak e=\sigma_{\mathrm{ess}}(\mathcal C_\alpha),
\]
where $\alpha=\{\alpha_n\}_{n\geq 0}\subset\mathbb D$ is periodic.

We begin with a standard consequence of Floquet theory \cite{SimonOPUC1,SimonOPUC2}.

\begin{lemma}\label{lem:periodic-linear-growth}
Let $\alpha=\{\alpha_n\}_{n\geq 0}\subset\mathbb D$ be periodic.
There exists a constant $C_0>0$ such that, for every
$0\leq m\leq n$ and every $z\in\mathfrak e$,
\[
\left\|
A_\alpha(n,z)A_\alpha(m,z)^{-1}
\right\|
\leq C_0(1+n-m).
\]
In particular,
\[
\|A_\alpha(n,z)\|
+
\|A_\alpha(n,z)^{-1}\|
\leq C_0(n+1),
\, \, n\geq 0,\, \,  z\in\mathfrak e.
\]
\end{lemma}

\begin{lemma}\label{lem:perturbed-linear-growth}
Let $\alpha=\{\alpha_n\}_{n\geq 0}\subset\mathbb D$ be periodic, and let
$b:\mathbb N_0\to\mathbb C$ satisfy
\[
M_1(b):=
\sum_{n=0}^{\infty}(n+1)|b(n)|<\infty.
\]
For
\[
\alpha_n^\kappa=\alpha_n+\kappa b(n),
\, \, n\geq 0,
\]
let $A_{\alpha^\kappa}(n,z)$ denote the corresponding Szeg\H{o}
transfer matrix. Then there exists $\varepsilon>0$ such that, for every
$|\kappa|<\varepsilon$, there exists $C>0$ satisfying
\[
\|A_{\alpha^\kappa}(n,z)\|
+
\|A_{\alpha^\kappa}(n,z)^{-1}\|
\leq C (n+1),
\, \, n\geq 0,\, \,  z\in\mathfrak e,
\]
where $C$ can be chosen uniformly for $\kappa$ in a sufficiently small neighborhood of the origin.
\end{lemma}

\begin{proof} Set
\[
B_j(z)=S(\alpha_j,z),
\, \, 
C_j(z)=S(\alpha_j^\kappa,z).
\]
Since $\sup_j|\alpha_j|<1$, by choosing $\varepsilon>0$ sufficiently
small we may assume that
\[
\sup_{j\geq 0}|\alpha_j^\kappa|\leq r<1
\]
for every $|\kappa|<\varepsilon$. Hence, there exists a constant
$C_1>0$, independent of $j$, $z$, and $\kappa$, such that
\begin{equation}\label{eq:one-step-polynomial}
\|C_j(z)-B_j(z)\|
\leq C_1|\kappa||b(j)|,
\, \, j\geq 0,\, \,  z\in\partial\mathbb D.
\end{equation}

Using the telescoping identity in the form
\[
A_{\alpha^\kappa}(n,z)
=
A_\alpha(n,z)
+
\sum_{j=0}^{n-1}
A_\alpha(n,z)A_\alpha(j+1,z)^{-1}
\bigl(C_j(z)-B_j(z)\bigr)
A_{\alpha^\kappa}(j,z),
\]
Lemma~\ref{lem:periodic-linear-growth} and
\eqref{eq:one-step-polynomial} give
\[
\|A_{\alpha^\kappa}(n,z)\|
\leq
C_0(n+1)
+
C_2|\kappa|
\sum_{j=0}^{n-1}
(n-j)|b(j)|
\|A_{\alpha^\kappa}(j,z)\|.
\]
Define
\[
X_n(z):=
\frac{\|A_{\alpha^\kappa}(n,z)\|}{n+1}.
\]
Since
\[
\frac{n-j}{n+1}\leq 1,
\]
we obtain
\[
X_n(z)
\leq
C_0
+
C_2|\kappa|
\sum_{j=0}^{n-1}
(j+1)|b(j)|X_j(z).
\]
The discrete Gronwall inequality therefore yields
\[
X_n(z)
\leq
C_0
\exp\left(
C_2|\kappa|
\sum_{j=0}^{\infty}(j+1)|b(j)|
\right) \leq
C_0
\exp\left(
C_2 \epsilon
\sum_{j=0}^{\infty}(j+1)|b(j)|
\right).
\]
Hence,
\[
\|A_{\alpha^\kappa}(n,z)\|
\leq C (n+1),
\]
uniformly in $n\geq 0$ and $z\in\mathfrak e$.

Finally, for $|z|=1$, the Szeg\H{o} transfer matrices have determinant
of modulus one. Hence, for the operator norm on $2\times2$ matrices,
\[
\|A_{\alpha^\kappa}(n,z)^{-1}\|
=
\|A_{\alpha^\kappa}(n,z)\|,
\]
and the result follows.
\end{proof}

\begin{proposition}\label{prop:asymptotic-formula-cmv}
Let $\alpha=\{\alpha_n\}_{n\geq 0}\subset\mathbb D$ be periodic and set
\[
\mathfrak e=\sigma_{\mathrm{ess}}(\mathcal C_\alpha).
\]
Assume that $b:\mathbb N_0\to\mathbb C$ satisfies
\begin{equation}\label{eq:second-moment-b}
M_2(b):=
\sum_{n=0}^{\infty}(n+1)^2|b(n)|<\infty.
\end{equation}
For
\[
\alpha_n^\kappa=\alpha_n+\kappa b(n),
\, \, n\geq 0,
\]
let $A_\alpha(n,z)$ and $A_{\alpha^\kappa}(n,z)$ denote the
corresponding Szeg\H{o} transfer matrices.

Then there exists $\varepsilon>0$ such that, for every
$|\kappa|<\varepsilon$, there exist continuous matrix-valued functions
\[
Q,P:\mathfrak e\longrightarrow M_2(\mathbb C)
\]
such that, for every $n\geq 0$ and $z\in\mathfrak e$,
\begin{equation}\label{eq:poly-comparison-I}
A_\alpha(n,z)-A_{\alpha^\kappa}(n,z)
=
A_\alpha(n,z)Q(z)+R_n(z),
\end{equation}
and
\begin{equation}\label{eq:poly-comparison-II}
A_{\alpha^\kappa}(n,z)-A_\alpha(n,z)
=
A_{\alpha^\kappa}(n,z)P(z)+K_n(z).
\end{equation}
Moreover, there exists a constant $C>0$, independent of $n$, $z$, and
$\kappa$, such that
\begin{equation*}\label{eq:QP-poly}
\|Q(z)\|+\|P(z)\|
\leq C|\kappa|M_2(b),
\, \, z\in\mathfrak e,
\end{equation*}
and
\begin{equation*}\label{eq:remainder-poly}
\|R_n(z)\|+\|K_n(z)\|
\leq
C|\kappa|(n+1)
\sum_{j=n}^{\infty}(j+1)^2|b(j)|.
\end{equation*}
\end{proposition}

\begin{proof}Let
\[
B_j(z)=S(\alpha_j,z),
\, \, 
C_j(z)=S(\alpha_j^\kappa,z).
\]
By taking $\varepsilon>0$ sufficiently small, there exists $C_1>0$ such
that
\begin{equation}\label{eq:step-difference-poly}
\|B_j(z)-C_j(z)\|
\leq
C_1|\kappa||b(j)|,
\, \, j\geq 0,\, \,  z\in\partial\mathbb D.
\end{equation}

The telescoping identity gives
\[
A_\alpha(n,z)-A_{\alpha^\kappa}(n,z)
=
A_\alpha(n,z)
\sum_{j=0}^{n-1}q_j(z),
\]
where
\[
q_j(z)
=
A_\alpha(j+1,z)^{-1}
\bigl(B_j(z)-C_j(z)\bigr)
A_{\alpha^\kappa}(j,z).
\]
By Lemmas~\ref{lem:periodic-linear-growth} and
\ref{lem:perturbed-linear-growth}, together with
\eqref{eq:step-difference-poly},
\begin{align*}
\|q_j(z)\|
&\leq
\|A_\alpha(j+1,z)^{-1}\|
\|B_j(z)-C_j(z)\|
\|A_{\alpha^\kappa}(j,z)\| \\
&\leq
C|\kappa|(j+2)(j+1)|b(j)| \\
&\leq
C|\kappa|(j+1)^2|b(j)|.
\end{align*}
Therefore, by \eqref{eq:second-moment-b},
\[
Q(z):=
\sum_{j=0}^{\infty}q_j(z)
\]
converges absolutely and uniformly for $z\in\mathfrak e$. In
particular,
\[
\|Q(z)\|
\leq
C|\kappa|
\sum_{j=0}^{\infty}(j+1)^2|b(j)|
=
C|\kappa|M_2(b).
\]
Since every $q_j$ is continuous in $z$, the uniform convergence also
implies that $Q$ is continuous on $\mathfrak e$.

Set
\[
E_n(z):=
\sum_{j=n}^{\infty}q_j(z).
\]
Then
\[
A_\alpha(n,z)-A_{\alpha^\kappa}(n,z)
=
A_\alpha(n,z)Q(z)-A_\alpha(n,z)E_n(z).
\]
Thus, setting
\[
R_n(z):=-A_\alpha(n,z)E_n(z),
\]
Lemma~\ref{lem:periodic-linear-growth} gives
\begin{align*}
\|R_n(z)\|
&\leq
\|A_\alpha(n,z)\|
\sum_{j=n}^{\infty}\|q_j(z)\| \\
&\leq
C|\kappa|(n+1)
\sum_{j=n}^{\infty}(j+1)^2|b(j)|.
\end{align*}
This proves \eqref{eq:poly-comparison-I} and the required estimate for
$R_n$. The proof of \eqref{eq:poly-comparison-II}, together with the required estimate for $K_n$, follows similarly. This completes the proof of the proposition.

\end{proof}

\section{Sharp Strichartz-type inequality}

Let $\nu$ be a positive Borel measure on the unit circle $\partial\mathbb{D}$,
absolutely continuous with respect to the normalized Lebesgue measure.
Let $f$ denote its Radon-Nikodym derivative. We say that $\nu$ has a
singularity of type $\varrho\in[0,1)$ if there exist a constant
$\gamma>0$, a finite family of pairwise disjoint arcs
$I_j\subset\partial\mathbb{D}$, $j=1,\ldots,m$, points
$z_j\in I_j$, numbers $0\leq\varrho_j\leq\varrho$, and non-negative
functions $f_j$ such that
\[
f(z)=\sum_{j=1}^m f_j(z),
\, \, 
0\leq f_j(z)\leq
\gamma f_{\varrho,j}(z)
=
\gamma\frac{\chi_{I_j}(z)}
{|z-z_j|^{\varrho_j}},
\, \, 
z\in\partial\mathbb{D}.
\]
When $\varrho_j=0$, the choice of $z_j$ is immaterial.

Furthermore, we say that a positive Borel measure $\mu$ on
$\partial\mathbb{D}$ has a weak singularity of type
$\varrho\in[0,1)$ if there exist a positive Borel measure $\nu$ with
a singularity of type $\varrho$ and a constant $\gamma'>0$ such that,
for every continuous non-negative function
$g:\partial\mathbb{D}\to\mathbb{R}$,
\[
\int_{\partial\mathbb{D}} g(z)\,d\mu(z)
\leq
\gamma'
\int_{\partial\mathbb{D}} g(z)\,d\nu(z).
\]

\begin{proposition}\label{thmStrichartz}
Let $\mu$ be as above. Suppose that $\mu$ has a weak singularity of
type $\varrho\in[0,1)$. Then there exists a constant $C_\mu>0$ such that:
\begin{enumerate}

\item[(i)] if $0\leq\varrho<\frac12$, then
\[
\frac{1}{K}
\sum_{n=0}^{K-1}
\left|
\int_{\partial\mathbb{D}}z^{-n}\,d\mu(z)
\right|^2
\leq
C_\mu\frac{1}{K},
\, \, 
\forall K\geq1;
\]

\item[(ii)] if $\varrho=\frac12$, then
\[
\frac{1}{K}
\sum_{n=0}^{K-1}
\left|
\int_{\partial\mathbb{D}}z^{-n}\,d\mu(z)
\right|^2
\leq
C_\mu\frac{\log(K)}{K},
\, \, 
\forall K\geq2;
\]

\item[(iii)] if $\frac12<\varrho<1$, then
\[
\frac{1}{K}
\sum_{n=0}^{K-1}
\left|
\int_{\partial\mathbb{D}}z^{-n}\,d\mu(z)
\right|^2
\leq
C_\mu\frac{1}{K^{2(1-\varrho)}},
\, \, 
\forall K\geq1.
\]

\end{enumerate}
\end{proposition}

\begin{proof}
For each $f\in L^1(\partial\mathbb{D})$, we denote its $n$-th
Fourier coefficient by
\[
\widehat{f}(n)
=
\int_{\partial\mathbb{D}}w^{-n}f(w)\,d\ell(w).
\]

For $0\leq n\leq K-1$, we have
$e^{1-n^2/K^2}\geq1$. Therefore,
\begin{eqnarray*}
\frac{1}{K}
\sum_{n=0}^{K-1}
\left|
\int_{\partial\mathbb{D}}z^{-n}\,d\mu(z)
\right|^2
&\leq&
\frac{e}{K}
\sum_{n\in\mathbb{Z}}
e^{-n^2/K^2}
\left|
\int_{\partial\mathbb{D}}z^{-n}\,d\mu(z)
\right|^2
\\
&=&
\frac{e}{K}
\int_{\partial\mathbb{D}}
\int_{\partial\mathbb{D}}
\left[
\sum_{n\in\mathbb{Z}}
e^{-n^2/K^2}z^{-n}w^n
\right]
d\mu(z)\,d\mu(w).
\end{eqnarray*}

The expression in brackets is the periodic Gaussian kernel. Indeed,
by the Poisson summation formula,
\[
\sum_{n\in\mathbb{Z}}
e^{-n^2/K^2}e^{in(\phi-\theta)}
=
\sqrt{\pi}K
\sum_{m\in\mathbb{Z}}
e^{-\frac{K^2}{4}(\phi-\theta-2\pi m)^2},
\, \, 
\phi,\theta\in[0,2\pi).
\]
In particular, this kernel is real-valued and non-negative.

Since the weak-singularity hypothesis implies
\[
d\mu\leq\gamma'\,d\nu
\]
in the sense of positive measures, and since the periodic Gaussian
kernel is non-negative, we may apply the domination successively in
the two variables. Thus
\begin{eqnarray}\label{eqFdec1}\nonumber
\frac{1}{K}
\sum_{n=0}^{K-1}
\left|
\int_{\partial\mathbb{D}}z^{-n}\,d\mu(z)
\right|^2
&\lesssim&
\frac{e}{K}
\int_{\partial\mathbb{D}}
\int_{\partial\mathbb{D}}
\left[
\sum_{n\in\mathbb{Z}}
e^{-n^2/K^2}z^{-n}w^n
\right]
d\nu(z)\,d\nu(w)
\\ \nonumber
&\lesssim&
\frac{e}{K}
\sum_{j,l=1}^m
\int_{\partial\mathbb{D}}
\int_{\partial\mathbb{D}}
\left[
\sum_{n\in\mathbb{Z}}
e^{-n^2/K^2}z^{-n}w^n
\right]
f_{\varrho,j}(z)
f_{\varrho,l}(w)
\,d\ell(z)\,d\ell(w)
\\ \nonumber
&=&
\frac{e}{K}
\sum_{j,l=1}^m
\sum_{n\in\mathbb{Z}}
e^{-n^2/K^2}
\widehat{f_{\varrho,j}}(n)
\overline{\widehat{f_{\varrho,l}}(n)}
\\
&\lesssim&
\frac{1}{K}
\sum_{n=1}^\infty
e^{-n^2/K^2}
n^{2(\varrho-1)}.
\end{eqnarray}
Here we have used the standard Fourier estimate for truncated power
singularities,
\[
\left|
\widehat{f_{\varrho,j}}(n)
\right|
\lesssim
(1+|n|)^{\varrho_j-1}
\leq
(1+|n|)^{\varrho-1},
\, \, 
n\in\mathbb{Z}.
\]
The finitely many indices $j$ and $l$ are absorbed into the implicit
constant.

We now estimate the last sum. If
$0\leq\varrho<\frac12$, then
$2(\varrho-1)<-1$, and hence
\[
\sum_{n=1}^\infty
e^{-n^2/K^2}n^{2(\varrho-1)}
\lesssim1.
\]
Consequently,
\[
\frac{1}{K}
\sum_{n=1}^\infty
e^{-n^2/K^2}n^{2(\varrho-1)}
\lesssim
\frac{1}{K}.
\]

If $\varrho=\frac12$, then
\[
2(\varrho-1)=-1,
\]
and therefore
\[
\sum_{n=1}^\infty
e^{-n^2/K^2}n^{-1}
\lesssim
\log(K),
\, \,  K\geq2.
\]
It follows that
\[
\frac{1}{K}
\sum_{n=1}^\infty
e^{-n^2/K^2}n^{2(\varrho-1)}
\lesssim
\frac{\log(K)}{K}.
\]

Finally, if $\frac12<\varrho<1$, then
$2(\varrho-1)>-1$. Splitting the sum at $n=K$ and using the
Gaussian decay for $n>K$, we obtain
\[
\sum_{n=1}^\infty
e^{-n^2/K^2}n^{2(\varrho-1)}
\lesssim
K^{2\varrho-1}.
\]
Hence,
\[
\frac{1}{K}
\sum_{n=1}^\infty
e^{-n^2/K^2}n^{2(\varrho-1)}
\lesssim
K^{-2(1-\varrho)}.
\]
Together with \eqref{eqFdec1}, these estimates prove
(i)--(iii).
\end{proof}

\begin{remark}{\rm 
The Gaussian regularization argument used above is analogous to the one developed in \cite{AloCardeO} for a related estimate involving the continuous-time Fourier transform.}
\end{remark}

\section{Proof of the Main Theorem}

We begin with a consequence of the asymptotic comparison formula established in Proposition~\ref{prop:asymptotic-formula-cmv}. The
purpose of the following lemma is to identify the asymptotic boundary vector which naturally arises from the perturbation.

\begin{lemma}\label{lem:ukappa-isotropic} Let $\mathfrak e = \sigma_{\mathrm{ess}}(\mathcal C_\kappa)$ and  let
\[
J=
\begin{pmatrix}
1&0\\
0&-1
\end{pmatrix},
\, \, 
u_0=\frac{1}{\sqrt2}
\begin{pmatrix}
1\\
1
\end{pmatrix}.
\]
Assume that $Q(z)$ is the matrix given by Proposition~\ref{prop:asymptotic-formula-cmv}
and that $\|Q(z)\|<1$. Define
\[
v(z):=(I-Q(z))u_0,
\, \, 
u_\kappa(z):=\frac{v(z)}{\|v(z)\|}.
\]
Then, for every $z\in\mathfrak e$,
\[
u_\kappa(z)^*Ju_\kappa(z)=0.
\]
\end{lemma}

\begin{proof}
For $|z|=1$, the Szeg\H{o} transfer matrices are $J$-unitary:
\[
A_\alpha(n,z)^*JA_\alpha(n,z)
=
A_{\alpha^\kappa}(n,z)^*JA_{\alpha^\kappa}(n,z)
=
J.
\]
By Proposition~\ref{prop:asymptotic-formula-cmv},
\[
A_{\alpha^\kappa}(n,z)
=
A_\alpha(n,z)(I-Q(z))-R_n(z).
\]
Hence,
\[
A_\alpha(n,z)^{-1}A_{\alpha^\kappa}(n,z)
=
I-Q(z)-A_\alpha(n,z)^{-1}R_n(z).
\]
For $z\in\mathfrak e^\circ$, periodic Floquet theory (see, e.g., \cite[Chapter~11]{SimonOPUC2}) gives uniform
boundedness of $A_\alpha(n,z)^{-1}$, while $R_n(z)\to0$ as $n \to \infty$. Therefore, for every $z\in\mathfrak e^\circ$,
\[
A_\alpha(n,z)^{-1}A_{\alpha^\kappa}(n,z)
\longrightarrow I-Q(z)
\]
as $n \to \infty$. Since every matrix on the left-hand side is $J$-unitary, passing to the
limit gives
\[
(I-Q(z))^*J(I-Q(z))=J.
\]
Since
\[
u_0^*Ju_0=0,
\]
we obtain
\[
v(z)^*Jv(z)
=
u_0^*(I-Q(z))^*J(I-Q(z))u_0
=
u_0^*Ju_0
=
0.
\]
Note that $v(z)\neq0$ because $\|Q(z)\|<1$, and therefore
\[
u_\kappa(z)^*Ju_\kappa(z)
=
\frac{v(z)^*Jv(z)}{\|v(z)\|^2}
=
0.
\]

Finally, since $u_\kappa$ is continuous on $\mathfrak e$, the identity extends by
continuity to the band edges.
\end{proof}

\begin{lemma}\label{lem:no-edge-atoms} Assume the hypotheses of Theorem~\ref{thm:perturbed-periodic-cmv}, and let $\mathfrak e =\sigma_{\mathrm{ess}}(\mathcal C_\alpha).$ If $|\kappa|$ is sufficiently small, then
\[
\mu_\alpha^\kappa(\{\zeta\})=0,
\, \,  \zeta\in \partial \mathfrak e.
\]
\end{lemma}

\begin{proof}
Let $\zeta\in\partial \mathfrak e$. By Proposition~\ref{prop:asymptotic-formula-cmv},
\[
A_{\alpha^\kappa}(n,\zeta)u_0
=
A_\alpha(n,\zeta)(I-Q(\zeta))u_0
-
R_n(\zeta)u_0,
\]
with
\[
\|R_n(\zeta)\| \leq \gamma |\kappa|(n+1)^{4-\tau},
\]
for some constant $\gamma>0$. For $|\kappa|$ sufficiently small, $||Q(\zeta)||<1$ and
\[
v_\kappa(\zeta):=(I-Q(\zeta))u_0\neq0.
\]
By Lemma~\ref{lem:ukappa-isotropic}, 
\[
u_\kappa(\zeta)
=
\frac{v_\kappa(\zeta)}{\|v_\kappa(\zeta)\|}
\]
satisfies  $u_\kappa(\zeta)^*Ju_\kappa(\zeta) =0.$ Hence, there exist $\eta,\lambda\in\partial\mathbb D$ such that
\[
u_\kappa(\zeta)
=
\eta\,\frac{1}{\sqrt2}
\binom{1}{\lambda}.
\]
As in Lemma~\ref{teclemma}, this gives
\[
\|A_\alpha(n,\zeta)u_\kappa(\zeta)\|
=
\|A_{\overline{\lambda}\alpha}(n,\zeta)u_0\|.
\]

The sequence $\overline{\lambda}\alpha$ is periodic and has the same discriminant as $\alpha$. Thus $\zeta$ is a band edge for the corresponding periodic CMV matrix, whose spectral measure has no point mass at $\zeta$. By the standard point-mass criterion for OPUC (Theorem 2.2.1 (ii) in \cite{{SimonOPUC1}}),
\[
\sum_{n=0}^{\infty}
\|A_{\overline{\lambda}\alpha}(n,\zeta)u_0\|^2
=
\infty.
\]
Moreover,
\[
\|v_\kappa(\zeta)\|
\|A_\alpha(n,\zeta)u_\kappa(\zeta)\|
\leq
\|A_{\alpha^\kappa}(n,\zeta)u_0\|
+
\|R_n(\zeta)u_0\|.
\]
Hence,
\[
\|v_\kappa(\zeta)\|^2
\|A_\alpha(n,\zeta)u_\kappa(\zeta)\|^2
\leq
2\|A_{\alpha^\kappa}(n,\zeta)u_0\|^2
+
2\|R_n(\zeta)u_0\|^2.
\]
Since $\|v_\kappa(\zeta)\|>0$ and
\[
\sum_{n=0}^\infty \|R_n(\zeta)u_0\|^2
<\infty,
\]
while
\[
\sum_{n=0}^\infty
\|A_\alpha(n,\zeta)u_\kappa(\zeta)\|^2=\infty,
\]
it follows that
\[
\sum_{n=0}^\infty
\|A_{\alpha^\kappa}(n,\zeta)u_0\|^2=\infty.
\]
The point-mass criterion therefore yields $\mu_\alpha^\kappa(\{\zeta\})=0.$
\end{proof}

\begin{theorem}\label{thm:comparison-spectral-measures-cmv} Assume the hypotheses of Theorem \ref{thm:perturbed-periodic-cmv}. Let $\mathfrak e = \sigma_{\mathrm{ess}}(\mathcal C_\kappa)$. Then, for every nonnegative function $f\in C(\mathfrak e)$,
\[
\int_{\mathfrak e} f(z)\,d\mu_\alpha^\kappa(z)
\leq
16
\int_{\mathfrak e} f(z)\,\rho_{\alpha,\kappa}(z)\,d\ell(z),
\]
where the density $\rho_{\alpha,\kappa}(z)$ has a singularity of order at most $\varrho=1/2$.
\end{theorem}

\begin{proof} Note that, by Weyl’s theorem \cite{SimonOPUC1}, one has
\[
\sigma_{\mathrm{ess}}(\mathcal C_\kappa)
=
\sigma_{\mathrm{ess}}(\mathcal C_\alpha)
=
\mathfrak e.
\]

Let
\[
u_0=\frac1{\sqrt2}
\begin{pmatrix}
1\\
1
\end{pmatrix}.
\]

Let $z\in \mathfrak e$ and let $\eta$ be given by Proposition \ref{prop:asymptotic-formula-cmv}. For $|\kappa|$ sufficiently small, by the same proposition, we have 
\[
A_\alpha(n,z)-A_{\alpha^\kappa}(n,z)
=
A_\alpha(n,z)Q(z)+R_n(z),
\]
where
\[
\|Q(z)\|\leq \eta|\kappa|,
\, \, 
\|R_n(z)\|\leq \eta |\kappa| (n+1)^{4-\tau},
\]
for some constant $\eta$ (which does not depend on $\kappa$). Choose $|\kappa|$ so small that $\eta|\kappa|<\frac1{100}.$ Then, for  every
$n\geq 0$,
\[
A_{\alpha^\kappa}(n,z)u_0
=
A_\alpha(n,z)(I-Q(z))u_0
-
R_n(z)u_0.
\]

Set
\[
v(z):=(I-Q(z))u_0.
\]
Since $Q$ is continuous and $\|Q(z)\|<1/100$, the map $v$ is continuous and
\[
\|v(z)-u_0\|\leq \|Q(z)\|<\frac1{100}.
\]
In particular,
\[\|v(z)\|\geq 1-\frac1{100}>\frac12.\] 

Set $u_\kappa(z):=\frac{v(z)}{\|v(z)\|}.$  Then $u_\kappa:\mathfrak e\to\mathbb C^2$ is continuous, $\|u_\kappa(z)\|=1$, and, by Lemma \ref{lem:ukappa-isotropic}, satisfies  $u_\kappa(z)^*Ju_\kappa(z)=0$. Moreover,
\[
A_\alpha(n,z)(I-Q(z))u_0
=
\|v(z)\|A_\alpha(n,z)u_\kappa(z).
\]
Therefore,
\[
A_{\alpha^\kappa}(n,z)u_0
=
\|v(z)\|A_\alpha(n,z)u_\kappa(z)
-
R_n(z)u_0.
\]
Hence, for every $z\in\mathfrak e$ and every $n\geq0$,
\[
\|A_{\alpha^\kappa}(n,z)u_0\|
\geq
\|v(z)\|\,
\|A_\alpha(n,z)u_\kappa(z)\|
-
\|R_n(z)\|.
\]
Since $\|v(z)\|\geq1/2$, we obtain
\[
\|A_{\alpha^\kappa}(n,z)u_0\|
\geq
\frac12\|A_\alpha(n,z)u_\kappa(z)\|
-
\|R_n(z)\|.
\]

We now show that the last term is negligible. Since $\alpha$ is periodic, the Szeg\H{o} transfer matrices $A_\alpha(n,z)$ have at most polynomial growth on the essential spectrum. More precisely, there exists $\gamma>0$ such that
\[
\|A_\alpha(n,z)\|+\|A_\alpha(n,z)^{-1}\|
\leq \gamma(n+1),
\, \,  z\in\mathfrak e,\ n\geq0.
\]
Therefore, for every unit vector $w\in\mathbb C^2$,
\[
\|A_\alpha(n,z)w\|
\geq
\|A_\alpha(n,z)^{-1}\|^{-1}
\geq
\frac1{\gamma(n+1)}.
\]
Applying this to $w=u_\kappa(z)$ gives
\[
\frac{\|R_n(z)\|}
{\|A_\alpha(n,z)u_\kappa(z)\|}
\leq
\gamma(n+1)\eta|\kappa| (n+1)^{4-\tau}.
\]
Since $\tau>5$, we have
\[
\gamma(n+1)\eta|\kappa| (n+1)^{4-\tau}\to 0
\]
as $n\to\infty$, uniformly in $z\in\mathfrak e$. Thus, there exists $n_0\geq1$ such that, for every $n>n_0$ and every $z\in\mathfrak e$,
\[
\frac{\|R_n(z)\|}
{\|A_\alpha(n,z)u_\kappa(z)\|}
\leq
\frac14.
\]
Hence,
\[
\|A_{\alpha^\kappa}(n,z)u_0\|
\geq
\frac12\|A_\alpha(n,z)u_\kappa(z)\|
-
\frac14\|A_\alpha(n,z)u_\kappa(z)\|.
\]
that is,
\[
\|A_{\alpha^\kappa}(n,z)u_0\|
\geq
\frac14\|A_\alpha(n,z)u_\kappa(z)\|,
\]
and therefore
\[
\|A_{\alpha^\kappa}(n,z)u_0\|^{-2}
\leq
16\|A_\alpha(n,z)u_\kappa(z)\|^{-2},
\, \,  n>n_0,\ z\in\mathfrak e.
\]

Now let $f\in C(\mathfrak e)$ be non-negative and let \(f\) be a continuous extension of \(f\) to \(\partial\mathbb{D}\), which we also denote by \(f\). We first prove the inequality for functions compactly supported in the interior of the essential spectrum. Assume that $f\in C_c(\mathfrak e^\circ), \, \,  f\geq0.$
By the Bernstein--Szegő approximation theorem (Theorem \ref{thm:cmv-carmona}),
\[
\int_{\partial\mathbb D} f(z)\,d\mu_\alpha^\kappa(z)
=
\lim_{n\to\infty}
\int_{\partial\mathbb D}
f(z)\|A_{\alpha^\kappa}(n,z)u_0\|^{-2}\,d\ell(z).
\]
Using the estimate above, we get
\[
\int_{\partial\mathbb D} f(z)\,d\mu_\alpha^\kappa(z)
\leq
16\lim_{n\to\infty}
\int_{\partial\mathbb D}
f(z)\|A_\alpha(n,z)u_\kappa(z)\|^{-2}\,d\ell(z).
\]
Since $f\in C_c(\mathfrak e^\circ)$ and $u_\kappa$ is continuous with
$\|u_\kappa(z)\|=1$, Lemma \ref{teclemma} gives
\[
\lim_{n\to\infty}
\int_{\partial\mathbb D}
f(z)\|A_\alpha(n,z)u_\kappa(z)\|^{-2}\,d\ell(z)
=
\int_{\partial\mathbb D}
f(z)\rho_{\alpha,\kappa}(z)\,d\ell(z).
\]
Hence,
\[
\int_{\partial\mathbb D} f(z)\,d\mu_\alpha^\kappa(z)
\leq
16
\int_{\partial\mathbb D}
f(z)\rho_{\alpha,\kappa}(z)\,d\ell(z),
\]
where, by the Lemma \ref{teclemma},   the density $\rho_{\alpha,\kappa}(z)$ has a singularity of order at most $\varrho=1/2$.

It remains to remove the restriction that $f$ be compactly supported in $\mathfrak e^\circ$. Let $\mathcal E$ denote the finite set of band edges of $\mathfrak e$.
For each $\varepsilon>0$, choose a continuous function
$\varphi_\varepsilon:\partial\mathbb D\to[0,1]$ such that
\[
\varphi_\varepsilon(z)=0
\, \, \text{if}\, \, 
\operatorname{dist}(z,\mathcal E)<\varepsilon,
\]
and
\[
\varphi_\varepsilon(z)=1
\, \, \text{if}\, \, 
\operatorname{dist}(z,\mathcal E)>2\varepsilon.
\]
and such that $\operatorname{supp}(\varphi_\varepsilon|_{\mathfrak e}) \Subset \mathfrak e^\circ$ and $\varphi_\varepsilon \to \chi_{\mathfrak e^\circ}$ as $\varepsilon\downarrow0$. Then $f_\varepsilon(z):=f(z)\varphi_\varepsilon(z)$ belongs to $C_c(\mathfrak e^\circ)$ when restricted to $\mathfrak e$, and $0\leq f_\varepsilon(z)\leq f(z).$

By Lemma~\ref{lem:no-edge-atoms}, \[ \mu_\alpha^\kappa(\partial \mathfrak e)=0. \] Hence $f_\varepsilon\to f$ $\mu_\alpha^\kappa$-almost everywhere on $\mathfrak e$. Since the set of band edges is finite and therefore has zero Lebesgue measure, the same convergence holds almost everywhere with respect to $\rho_{\alpha,\kappa}\,d\ell$. 

Applying the inequality already proved to $f_\varepsilon$, we get
\[
\int_{\partial\mathbb D}
f_\varepsilon(z)\,d\mu_\alpha^\kappa(z)
\leq
16
\int_{\partial\mathbb D}
f_\varepsilon(z)\,\rho_{\alpha,\kappa}(z)\,d\ell(z).
\]
But, by dominated convergence, we also obtain
\[
\int_{\mathfrak e}
f(z)\,d\mu_\alpha^\kappa(z)
\leq
16
\int_{\mathfrak e}
f(z)\,\rho_{\alpha,\kappa}(z)\,d\ell(z),
\]
which completes the proof.
\end{proof}

\subsection{Proof of Theorem~\ref{thm:perturbed-periodic-cmv}}

\noindent $(i)$ Recall that, by Weyl’s theorem \cite{SimonOPUC1},  $\sigma_{\mathrm{ess}}(\mathcal C_\kappa) = \sigma_{\mathrm{ess}}(\mathcal C_\alpha) = \mathfrak e.$ 

By Theorem \ref{thm:comparison-spectral-measures-cmv}, for every
non-negative function $f\in C(\mathfrak e)$,
\[
\int_{\mathfrak e} f(z)\,d\mu_{\alpha^\kappa}(z)
\leq
16
\int_{\mathfrak e} f(z)\,\rho_{\alpha,\kappa}(z)\,d\ell(z),
\]
where the measure $d\mu_{\alpha,\kappa} (z) = \rho_{\alpha,\kappa}(z)\,d\ell(z)$ has a weak singularity of type $\varrho=1/2$.
Equivalently,
\[
\mu_{\alpha^\kappa}|_{\mathfrak e}
\leq
16 \mu_{\alpha,\kappa}|_{\mathfrak e}
\]
in the sense of positive Borel measures. It follows that $\mu_{\alpha^\kappa}|_{\mathfrak e}$ is also absolutely continuous with respect to Lebesgue measure on $\partial\mathbb D$. Moreover, the vector $\delta_0$ is cyclic for the half-line CMV matrix $\mathcal C_\kappa$. Therefore, the spectral type of $\mathcal C_\kappa$ on $\mathfrak e$ is determined by the canonical
spectral measure $\mu_{\alpha^\kappa}$. Since this measure has no singular component on $\mathfrak e$, the restriction of $\mathcal C_\kappa$ to $\mathfrak e$ is purely absolutely continuous. This proves item $(i)$.

\noindent $(ii)$ By Lemma \ref{teclemma}, $\mu_{\alpha,\kappa}$ is purely absolutely continuous on $\mathfrak e$. Moreover, for each band edge
$\zeta\in\partial\mathfrak e$, there exist a neighborhood $U_\zeta$ of $\zeta$ and a constant $C_\zeta>0$ such that
\[
\rho_{\alpha,\kappa}(z)
\leq
\frac{C_\zeta}{|z-\zeta|^{1/2}},
\, \,  z\in U_\zeta\cap\mathfrak e^\circ.
\]
Since there are only finitely many band edges and the density is bounded away from them, there exists a constant $\gamma>0$ such that, for every arc $I\subset\partial\mathbb D$ with $\ell(I)<1$,
\[
\mu_{\alpha,\kappa}(I\cap\mathfrak e)
\leq
\gamma\,\ell(I)^{1/2}.
\]
Namely, near a band edge $\zeta$, one estimates
\[
\int_{I}
\frac{d\ell(z)}{|z-\zeta|^{1/2}}
\leq
\gamma\,\ell(I)^{1/2}.
\]
Therefore, 
\[
\mu_{\alpha^\kappa}(I\cap\mathfrak e)
\leq
\gamma \mu_{\alpha,\kappa}(I\cap\mathfrak e)
\leq
\gamma \,\ell(I)^{1/2}.
\]
Thus $\mu_{\alpha^\kappa}|_{\mathfrak e}$ is $1/2$-H\"older continuous. This
proves item $(ii)$.

\noindent $(iii)$ The density estimate obtained above means precisely that the restriction of $\mu_{\alpha^\kappa}$ to $\mathfrak e$ is dominated by
a measure of singularity type $1/2$ in the sense of Section~3. Therefore, Proposition~\ref{thmStrichartz}, with
$\varrho=1/2$, gives
\[
\frac1K
\sum_{n=0}^{K-1}
\left|
\int_{\mathfrak e}z^n\,d\mu_{\alpha^\kappa}(z)
\right|^2
\lesssim \frac{\log K}{K}.
\]

Finally, by the spectral theorem and the definition of
$\mu_{\alpha^\kappa}$ as the spectral measure of $\delta_0$,
\[
\left\langle
\chi_{\mathfrak e}(\mathcal C_{\alpha^\kappa})\delta_0,
\mathcal C_{\alpha^\kappa}^n
\chi_{\mathfrak e}(\mathcal C_{\alpha^\kappa})\delta_0
\right\rangle
=
\int_{\mathfrak e}
z^n\,d\mu_{\alpha^\kappa}(z).
\]
Hence,
\[
\frac1K
\sum_{n=0}^{K-1}
\left|
\left\langle
\chi_{\mathfrak e}(\mathcal C_{\alpha^\kappa})\delta_0,
\mathcal C_{\alpha^\kappa}^n
\chi_{\mathfrak e}(\mathcal C_{\alpha^\kappa})\delta_0
\right\rangle
\right|^2
\lesssim \frac{1}{\|\chi_{\mathfrak e}(\mathcal C_{\alpha^\kappa})\delta_0\|^4} \frac{\log K}{K}.
\]
This is the asserted time-averaged dynamical estimate.


 
\begin{center} \Large{Acknowledgments} 
\end{center}
\addcontentsline{toc}{section}{Acknowledgments}

\noindent M. Aloisio was supported by grant \#2025/25338-1 from the São Paulo Research Foundation (FAPESP) and in part by grant \#01/24/APQ-03132-24 from the Minas Gerais Research Foundation (FAPEMIG). C. R. de Oliveira acknowledges partial support from CNPq (Brazilian National Council for Scientific and Technological Development) under Grant No. 316897/2025-6.


\noindent   Moacir Aloisio. Email: moacir.aloisio@ufvjm.edu.br, DME, UFVJM, Diamantina, MG, 39100-000, Brazil and ICMC, USP, S\~ao Carlos, SP, 13566-590, Brazil

\noindent  C\'esar R. de Oliveira. Email: oliveira@ufscar.br,  DM,   UFSCar, S\~ao Carlos, SP, 13560-970 Brazil

\noindent  Mariane Pigossi. Email: mariane.pigossi@ufes.br, DMAT, UFES, Vit\'oria, ES, 29075-910 Brazil

\end{document}